\documentclass{amsart}
\usepackage{amsmath,amssymb,amsthm,amscd,amsfonts,amsopn,amsxtra}
\usepackage{latexsym,array,hyperref,tensor,centernot}
\usepackage[mathscr]{eucal}
\usepackage{graphicx,color,fancybox,shadow,epic}
\usepackage{epsfig,pstricks,pst-grad, pst-text,
pst-optic}
 \usepackage{mathrsfs}

\newtheorem{theorem}{Theorem}[section]
\newtheorem{lemma}[theorem]{Lemma}
\newtheorem{proposition}[theorem]{Proposition}

\newtheorem{hypothesis}{Hypothesis}

\theoremstyle{definition}
\newtheorem{definition}[theorem]{Definition}

\theoremstyle{remark}
\newtheorem{remark}[theorem]{Remark}

\numberwithin{equation}{section}

\newcommand{\be}{\begin{equation}}
\newcommand{\ee}{\end{equation}}

\definecolor{orange}{rgb}{0.995, 0.75, 0.35}
\definecolor{purple}{rgb}{0.7, 0.2, 0.5}
\definecolor{royalblue}{rgb}{0.2, 0.7, 0.8}
\definecolor{darkgreen}{rgb}{0.2,0.725,0.25} 

\newcommand{\norm}[1]{\|#1\|}

\def\de{\delta}
\def\eps{\epsilon}
\def\ga{\gamma}
\def\lam{\lambda}

\def\vphi{\varphi}

\def\De{\Delta}
\def\Ga{\Gamma}

\def\inv{^{-1}}
\def\iy{\infty}
\def\la{\langle}
\def\ra{\rangle}

\newcommand{\R}{\mathbb{R}}

\begin{document}
\title[Orbital Stability  for fractional Hartree equation]{Orbital Stability of Standing Waves  for  Fractional 
Hartree  Equation  with Unbounded Potentials}


\author{Jian Zhang} 
\address[Jian Zhang]{School of Mathematical Science\\
 University of Electronic Science and Technology of China\\
  Chengdu, Sichuan, 611731, China}  
\email{zhangjiancdv@sina.com} 
\author{Shijun Zheng}
\address[Shijun Zheng]{Department of Mathematical Sciences, Georgia Southern University\\ 
Statesboro, Georgia 30460-8093, USA}  
\email{szheng@GeorgiaSouthern.edu}
\author{Shihui Zhu}
\address[Shihui Zhu]{Department of Mathematics, Sichuan Normal University\\
 Chengdu, Sichuan, 610066, China.} 
\email{shihuizhumath@163.com}

\thanks{}


\date{}

\begin{abstract} 
  We prove the existence of the set of ground states in a suitable energy space $\Sigma^s=\{u: \int_{\R^N} \bar{u}(-\De+m^2)^s  u+V |u|^2<\iy\}$,
$s\in (0,\frac{N}{2})$
for the mass-subcritical nonlinear fractional Hartree equation with unbounded potentials. 
As a consequence we  obtain, 
 as a priori result, the orbital stability of the set of 
  standing waves.  The main ingredient is the observation that $\Sigma^s$ is compactly embedded in $L^2$. This enables us to apply the concentration compactness argument in 
  the works of Cazenave-Lions and Zhang, 
  namely, relative compactness for any minimizing sequence in the energy space.
\end{abstract}

\maketitle

\section{Introduction} 
 Consider the nonlinear fractional  Hartree equation with an unbounded potential  in the following form: 
For $0<s<N/2$ and $(t,x)\in \mathbb{R}^{1+N}$ 
\begin{align}
&iu_t=(-\De+m^2)^su+V(x)u-(\frac{1}{|x|^{\gamma}}*|u|^{2})u\label{1.1}\\ 
&u(0,x)=u_0\in \Sigma^s, \label{1.2}
\end{align}
where  $u=u(t,x)$: $\mathbb{R}\times \mathbb{R}^N \to \mathbb{C}$ is  a complex valued function and the convolution $*$ is defined by 
$W*|u|^2:=\left(\frac{1}{|\cdot|^{\gamma}}*|u|^2 \right) (x)=\int \frac{|u(y)|^2}{|x-y|^{\gamma}}dy$ with $0<\gamma<\min\{4s,N\}$. 
The operator $(-\De+m^2)^{s}$ is defined by
\[(-\De+m^2)^{s}u=\mathcal{F}^{-1}[ (|\xi|^{2}+m^2)^s\mathcal{F}[u](\xi)], \]
where $m\geq 0$, $\mathcal{F}$ and
$\mathcal{F}^{-1}$ are the Fourier transform and the inverse
 in $\mathbb{R}^N$, respectively.  
  Here  the potential function $V\in C^\iy(\R^N)$ is bounded from below $V(x)\geq -c_0$
   for some $c_0>0$ and satisfies
\begin{equation}\label{Potential} 
 V(x)\rightarrow \infty\ \ \ {\rm as}\ \ \ |x|\rightarrow \infty.
\end{equation} 
 Such class of $V$ includes  unbounded potentials arising in physics, for instance,  the harmonic potential $V(x)=|x|^2$, 
and more generally,  
 polynomial functions  that are bounded from below. 
  Then  $H_{s,V}:=(-\Delta+m^2)^s+V $ is essentially self-adjoint in the Hilbert space 
  $\Sigma^s:= \{v\in L^2: (-\De+m^2)^{s/2}v\in L^2 \; and \;|V|^{1/2} v\in L^2\}$, 
  as is given in \mbox{Section \ref{s:prelim}.}

The Hartree equation is non-integrable analog of the NLS system, which 
arises naturally in large quantum systems that describe the wave motion for the bosonic or fermionic particles. 
When $s=1$, it can be derived from the mean-field limit of $N$-body GPE hierarchy, where $W(x)=|x|^{-\ga}$ 
represents the long-range two-body  interaction potential. 
A special feature of the Hartree equation lies in the convolution kernel $W$ 
that preserves the fine structure of micro two-body interactions of particles.
This is in contrast to the NLS, which can be viewed as the limiting case $W\to \de$ and where the two-body interactions are modeled from the scattering length. 
When $s=\frac12$, (\ref{1.1})  arises as an effective description of pseudo-relativistic
boson stars  in the mean field limit,  where $u(t, x)$ is a complex-valued wave field. 
The study of  such a perturbed system in the presence of external potential $V$ is both physically and mathematically very important. 

When $V$ is zero or  bounded,  the dynamical properties of solutions for   
the fractional Hartree equation \eqref{1.1}  have been considered in e.g., \cite{LieYau87,FrohlichLenzmann2007, HainzlLenzmannLewinSchlein2010, BaoDong2011,zhangzhu2015,FengZha18,guozhu18,zhu2016}.   
 In \cite{zhangzhu2015,FengZha18}, the authors studied the orbital stability for standing wave solutions, that is, 
 ground states  (the set of minimizers)  
   for  (\ref{1.1})  with zero potential by  profile decomposition method. Furthermore, \cite{zhangzhu2015} showed the strong instability in the mass-critical case. 
The  analogous results on orbital stability for nonlinear fractional Schr\"{o}dinger equation (FNLS) were obtained in 
 \cite{CHHO2014,Guohuang2012,zhu2017}. 
 The paper \cite{CHHO2014} studied the  FNLS
with power nonlinearity and a decaying potential, 
  while \cite{Guohuang2012,zhu2017} studied 
   FNLS  with power nonlinearity in the form $|u|^{p-1}u+|u|^{q-1}u$.
  
In this paper, 
we are concerned with the orbital stability of standing wave solutions for (\ref{1.1}) 
in  the  mass-subcritical regime  $\ga<2s$. 
Note that when $V=0$,   $\ga=2s$ corresponds to the 
   mass-critical case and $\ga=4s$ 
corresponds to the energy-critical case   by scaling invariance argument.  
The presence of an unbounded $V$ brings in technical difficulties 
and  the methods in \cite{CHHO2014,Guohuang2012,zhangzhu2015} do not  directly apply to \eqref{1.1}.  
Motivated by  the treatment for the classical NLS in \cite{zhang2000a,Zhang00}, we prove  a new compactness lemma in Section \ref{s:proof-main}, adapted to the energy  space  for  (\ref{1.1}), which leads to the solution to the variational problem \ref{VP} in Proposition \ref{p:fnlsV}.
Thus follows the main result Theorem \ref{th:orb-stab_V}.  The proof in Section \ref{s:orbital-hartreeV} 
can be viewed as an adaption to the potential case,
also see  \cite{CazLion82,Cazenave2003} for the original treatment of NLS type equations via variational method.  

 \section{Preliminaries}\label{s:prelim}
  We will use the notations $L^q:=L^q(\mathbb{R}^N)$, $\|\cdot\|_q:=\|\cdot\|_{L^q(\mathbb{R}^N)}$,
$H^s:=H^s(\mathbb{R}^N)$ and $\dot{H}^s:=\dot{H}^s(\mathbb{R}^N)$, the latter two denoting the usual Sobolev space and its homogeneous version. 
The various positive constants will be simply denoted by $C$.   From now on throughout the paper, without loss of generality  we may assume  there is some positive constant $c_1$ such that   for all $x$,
\begin{align}\label{mod-potential}
V(x)\ge c_1>0\quad and\quad\; V(x)\to \iy\quad as\; |x|\to \iy.
\end{align}
If otherwise,  the lower bound of $V$ is $-c_0<0$, then one can always 
 apply the substitution $u\to e^{itC}u$ for any  $C\ge c_0+c_1$ to convert  equation \eqref{1.1} into one with $V$ bounded from below by a positive constant; 
see also Remark \ref{re:genV}. 

 For   $V$ satisfying (\ref{mod-potential}) define the energy space  $\Sigma^s$  as
 \begin{equation}\label{e:Sigma^s_V}
 \Sigma^s:=\{ v\in L^2\big\vert\  \int \overline{v}(-\De+m^2)^s  v+V |v|^2<\infty\}.
 \end{equation}
 Then  $\Sigma^s$ is a Hilbert space equipped with the norm
  \begin{equation}\label{v-Hs-norm}
\|v\|_{\Sigma^s}:=\left(\int \overline{v}(-\De+m^2)^s  v+V |v|^2\right)^{\frac12},
\end{equation}
which means that $H_{s,V}=(-\De+m^2)^s +V$ is
 essentially self-adjoint   in the quadratic form on $\Sigma^s\times \Sigma^s$ given by 
\begin{align*}
&q(u,v):=
 \int\overline{(-\De+m^2)^{s/2} v} \,(-\De+m^2)^{s/2}u+\int \overline{v}Vu.
\end{align*}
Note that $\|v\|_{\Sigma^s}<\iy$  implies  $\norm{v}_2$ is finite. 
 When $s\in (0,1)$ and $m=0$, one can also express
 \begin{align*}
(-\De)^s u=c_{s,N}\, p.v.\, \int_{\R^N} \frac{u(x)-u(y)}{|x-y|^{N+2s}}dy,
\end{align*}
where $c_{s,N}=\frac{2^{2s}}{\pi^{N/2}} \frac{\Ga(s+\frac{N}{2})}{ |\Ga(-s)|}$. 
 
   Define the energy functional $E: \Sigma^s\to \R$ as
 \begin{equation}\label{Ev:DesV_hartree}
  E(v)= \frac12\int \overline{v}(-\De+m^2)^s  v+\frac 12 \int V|v|^2
 -\frac{1}{4} \int (\frac{1}{|x|^{\gamma}}*|v|^2) |v|^2.
 \end{equation}
From the Hardy inequality (\ref{Hs-hardy}),   the functional  $E(v)$ is
 well-defined in $\Sigma^s$.  Note that  if $V=0$ and $m=0$, $\ga=4s<N$ 
 corresponds to  the energy-critical case due to that the scaling 
$u\mapsto \frac{1}{\lam^{(N-2s)/2}}u(\frac{t}{\lam^{2s}},\frac{x}{\lam} )$ leaves invariant the energy and the solution to \eqref{1.1}. 
 When $\ga=2s<N$, it corresponds to 
 to  the mass-critical case due to that the scaling 
$u\mapsto \frac{1}{\lam^{N/2}}u(\frac{t}{\lam^{2s}},\frac{x}{\lam} )$ leaves invariant the mass $M(u):=\int |u|^2$ and \eqref{1.1}. 
 
 In this paper, we assume that the Cauchy problem (\ref{1.1})-(\ref{1.2}) is 
 well-posed in
 $\Sigma^s$, namely, the following global in time existence, uniqueness 
 and conservation laws hold in the mass-subcritical regime $\ga<2s$. 
\begin{hypothesis}\label{h:lwp-blupaltern-conserv} 
Let $N\geq 1$ and $\frac{\gamma}{2}<s<N/2$. If the
initial data $u_{0}\in \Sigma^s$, then
 there exists a unique solution $u(t,x)$
of the Cauchy problem (\ref{1.1})-(\ref{1.2}) on $\R$ 
such that $u\in C(\R;\Sigma^{s})\bigcap C^1(\R;H^{-s})$. 
Moreover, for all $t\in \R$, $u(t,x)$ satisfies the following conservation laws: 
\begin{itemize}
  \item [(i)]   \textup{(mass)}  \ \ \[  M(u(t))=M(u_0) .\]
  \item [(ii)] \textup{(energy)}\ \ \  \[E(u(t))=E(u_{0}).\]
\end{itemize}
\end{hypothesis}

\begin{remark} For local well-posedness in the above conjectured proposition, so far we only know about numerical result on cubic FNLS with a quadratic  potential by Zhang and Kirkpatrick \cite{KirZh16}, where 
  is suggested the long-time existence along with its dynamics for some special initial data. 
The theoretical proof has remained an {\em open question} concerning FNLS and Hartree equations with harmonic potentials. 
The main difficulty is the lack of proper dispersive or Strichartz estimates because the fractional Laplacian for $0<s<1$ does not hold a control  over the harmonic potential,
which  is shown in the deformed trajectories for the associated hamiltonian, see \cite{KirZh16} and \cite{SeSq14}.
This is in sharp contrast to the classical NLS $(s=1)$ with a harmonic potential, where 
the existence and stability problem has been studied quite extensively \cite{Fukuizumi2001,FukuizumiOhta2003,RoseWeinstein1988,zhang2000a,Zhang00,HuangZhangLi13,Feng16}. 

Heuristically the Laplacian $-\De$ and $V=|x|^2$ have balanced strength or effect so the $L^1\to L^\iy$ time decay $t^{-N/2}$ holds locally for $e^{it(\De-V)}$. 
However, for $0<s<1$,
in the phase space the bound energy for fractional laplacian $(-\De+m^2)^s$ is less than  $|x|^2$, 
or, $H_{s,0}$ relative to $V$ is like the laplacian vs. anharmonic potential of higher order. 
It falls within the quantum situation  
``a particle at a higher altitude falls down to the bottom of the potential in 
a shorter time than one at a lower altitude'', which  obstructs Fujiwara's  theorem,   
cf. \cite{Oh89}. 
  \end{remark}

\section{Proof of main result}\label{s:proof-main}
For estimating the Hartree nonlinearity we will need Hardy's inequality, see e.g, \cite{Tao2006}: If $s\in (0,N/2)$
\begin{align}
&\sup_{x}\int \frac{|u(y)|^2}{|x-y|^{2s}}dy\le c(s,N)\Vert u\Vert^2_{\dot{H}^{s}} \,.\label{Hs-hardy}
\end{align}

\begin{lemma}\label{l:ga-hardy-hartree} Let $v\in \dot{H}^s$ and $0<s<N/2$. If $0<\gamma< 2s$, then  there exists a positive constant $C>0$ such that  
\be\label{G-N}
\int(\frac{1}{|x|^\gamma}*|v|^2)|v|^2dx\leq
C \|v\|_{\dot{H}^s}^{\frac{\gamma}{s}}\
\|v\|_2^{\frac{4s-\gamma}{s}}.
\ee
\end{lemma}
\begin{proof} Note that 
\be\label{xga-v4}
\begin{array}{lll}
\int (\frac{1}{|x|^\gamma}*|v|^2)|v(x)|^2dx&=\int\int\frac{|v(x)|^2}{|x-y|^\gamma}|v(y)|^2dxdy\\
&\leq
 \|\int\frac{|v(x)|^2}{|x-y|^\gamma}dx \|_{\infty}\|v(y)\|_2^2.\end{array}\ee
Using  H\"{o}lder inequality with
  $1=\frac{\gamma}{2s}+\frac{2s-\gamma}{2s}$, we obtain 
  \begin{align}\label{e:hardy-ga-s2}
  \int\frac{|v(x)|^2}{|x-y|^\gamma}dx&=\int\frac{|v(x)|^{\frac{\gamma}{s}}}{|x-y|^\gamma}|v(x)|^{\frac{2s-\gamma}{s}}dx\notag\\
& \leq C \||x-y|^{-s}v(x)\|_2^{\frac{\gamma}{s}} \
\|v\|_2^{\frac{2s-\gamma}{s}}\notag\\
&\leq C \|v\|_{\dot{H}^s}^{\frac{\gamma}{s}}\
\|v\|_2^{\frac{2s-\gamma}{s}}.
\end{align}
In the last step of the above, we have employed  \eqref{Hs-hardy}. 
Thus, (\ref{G-N}) follows
from (\ref{xga-v4})  and (\ref{e:hardy-ga-s2}).
\end{proof}

\begin{lemma}\label{l:vn-hartree-ga}
Let $0<\gamma <2s$, $0<s<N/2$ and $V$ satisfy \eqref{mod-potential}.
 Suppose $\{v_n\}_{n=1}^{\infty}$  converges weakly to  $U$ in $\Sigma^s$. 
 Then there exists a subsequence (still denoted by $\{v_n\}$) such that
   \be\label{Compact1}
\|v_{n}\|_2^2\rightarrow \|U\|_2^2\quad {\rm as }\quad n\rightarrow \infty,
\ee and
   \be\label{Compact2}
\iint \frac{|v_{n}(x)|^2|v_{n}(y)|^2}{|x-y|^{\gamma}}dxdy\rightarrow \iint \frac{|U(x)|^2|U(y)|^2}{|x-y|^{\gamma}}dxdy\ \ \ {\rm as }\quad n\rightarrow \infty.
\ee
\end{lemma}
\begin{remark} Equation (\ref{Compact1}) indeed implies that for $s>0$, $\Sigma^s$ is compactly embedded in $L^2$.
\end{remark}
\begin{proof}  Since 
 $\Sigma^s\subset L^2$, it is easy to see that  for $U\in \Sigma^s\cap L^2$, 
$v_n\rightharpoonup U$ weakly in $\Sigma^s$ and 
\be\label{Lemma21} 
 v_n\rightharpoonup U \ \  {\rm weakly \ \ in }\ \ \ \ L^2
 \ee
as $n \rightarrow \infty$. 
For an elementary proof of \eqref{Lemma21}, see Proposition \ref{p:vn-wc-f-g} in the Appendix.

Since $\{v_n\}$ is weakly convergent in $\Sigma^s$, $ \Vert v_n\Vert_{\Sigma^s}$ is uniformly bounded. So there is 
a positive constant $K$ such that 
 \begin{equation}\label{e:Vvn-Sigma}
 \sup\limits_{n} \int V(x)|v_n(x)|^2dx\le \sup_n\Vert v_n\Vert^2_{\Sigma^s}<K.
 \end{equation}
 Furthermore, we will show that there is a subsequence of $\{v_n\}$ that  strongly converges in $L^2$ and satisfies \eqref{Compact2}. 

 (1) First, we consider the case where $U=0$. From \eqref{e:Vvn-Sigma}  with $V$ satisfying (\ref{mod-potential}), 
we have for arbitray $\epsilon>0$, there exists a constant $B=B_\eps>0$ (large enough) such that $\frac{1}{V(x)}\leq \epsilon$ when $|x|>B$. We see that 
\[
\int_{|x|> B} |v_n|^2dx=  \int_{|x|> B}\frac{1}{V(x)}\ V(x) |v_n|^2dx\leq K\epsilon.
\] 
For the fixed $B$,  compact embedding property for Sobolev space on bounded domain gives  that one can extract a subsequence
(still denoted by $\{v_n\}$) such that 
\be\label{Lemma22}  
v_n\rightarrow 0 \qquad {\rm strongly \; in }\;   L^2(\{|x|\leq B\}).
\ee
Hence,  there exists $L=L_\eps$ such that for all $n>L$, 
$\int_{|x|\leq B} |v_n|^2dx\leq \epsilon$. Thus, we have, if $n>L$,
\be\label{Lemma23}\int |v_n|^2dx=\int_{|x|\leq B} |v_n|^2dx+\int_{|x|> B} |v_n|^2dx\leq (K+1)\epsilon.
\ee
Now, taking $\eps=\eps_k\to 0$ as $k\to \iy$, a standard diagonal argument shows that 
there exists a subsequence of $\{v_n\}$, still denoted by $\{v_n\}$, such that 
   (\ref{Compact1}) is true  
   in the case $U=0$. 
Then, it follows from  (\ref{G-N}) in Lemma \ref{l:ga-hardy-hartree}  that 
\[
\int\int \frac{|v_n(x)|^2|v_n(y)|^2}{|x-y|^{\gamma}}dxdy \leq C \|v_n\|_{\dot{H}^s}^{\frac{\gamma}{s}}\
\|v_n\|_2^{\frac{4s-\gamma}{s}}\rightarrow 0\  \ \ \ {\rm as }\ \ \ \  n\rightarrow \infty.
\] 
This proves  (\ref{Compact2})  when $U=0$. 

(2) Secondly, we consider the case $U\neq 0$. Take $w_n=v_n-U$. We have
 $w_n\rightharpoonup 0$ weakly in $\Sigma^s$ and in $L^2$ as $n \rightarrow \infty$. Then, from the above discussion, 
 we see that there is a subsequence (still denoted by $\{w_n\}$) such that 
  $w_n\rightarrow 0 $  strongly   in  $L^2$. In other words, we have
\be\label{Lemma34}
\|v_n-U\|_2\rightarrow 0\ \  \ {\rm as}\ \ \ n\rightarrow\infty.
\ee
  To prove (\ref{Compact2}),  we will apply \eqref{e:hardy-ga-s2}. 
  Indeed, if $0<\gamma<2s$, we deduce that 
\begin{align*}
  &\left| \iint \frac{|v_n(x)|^2|v_n(y)|^2}{|x-y|^{\gamma}}dxdy- \iint \frac{|U(x)|^2|U(y)|^2}{|x-y|^{\gamma}}dxdy\right|\\
\leq &\left| \iint \frac{|v_n(x)|^2|v_n(y)|^2}{|x-y|^{\gamma}}dxdy-\iint \frac{|v_n(x)|^2|U(y)|^2}{|x-y|^{\gamma}}dxdy\right|\\
& +\left| \iint \frac{|v_n(x)|^2|U(y)|^2}{|x-y|^{\gamma}}dxdy-\iint \frac{|U(x)|^2|U(y)|^2}{|x-y|^{\gamma}}dxdy\right|\\
\leq & \|\int\frac{|v_n(x)|^2}{|x-y|^\gamma}dx \|_{\infty} \left|\|v_n\|_2^2-\|U\|_2^2\right|+\|\int\frac{|U(y)|^2}{|x-y|^\gamma}dy \|_{\infty}
\left|\|v_n\|_2^2-\|U\|_2^2\right|\\
\leq &C \|v_n\|_{\dot{H}^s}^{\frac{\gamma}{s}}\
\|v_n\|_2^{\frac{2s-\gamma}{s}}\left|\|v_n\|_2^2-\|U\|_2^2\right|+\|U \|_{\dot{H}^s}^{\frac{\gamma}{s}}\
\|U\|_2^{\frac{2s-\gamma}{s}}\left| \|v_n\|_2^2-\|U\|_2^2 \right|\\
\leq &C\left|\|v_n\|_2^2-\|U\|_2^2\right| \rightarrow 0\ \ {\rm as }\ \ n\rightarrow \infty.
\end{align*}
Therefore, (\ref{Compact2}) is true for $U\neq 0$, which  completes the proof.
\end{proof}

\section{Orbital stability of standing waves}\label{s:orbital-hartreeV} 
Given $0<\gamma<2s$, $s\in (0,N/2)$, $N\ge 1$, let  $m\geq 0$  and $M>0$. Consider  the following variational problem 
\be\label{VP}
d_M:=\inf\limits_{\{v\in \Sigma^s\big\vert \, \|v\|_2^2=M\}} E(v),
\ee 
where 
$E(v)= \frac 12 \norm{v}_{\Sigma^s}^2-\frac14\iint \frac{|v(x)|^2 |v(y)|^2}{|x-y|^{\gamma}} dxdy$, 
as defined in \eqref{Ev:DesV_hartree}. 
  The  following proposition constructs a minimizer, called ground state, to the problem (\ref{VP}). 
 \begin{proposition}\label{p:fnlsV} 
 Let  $M>0$.  Suppose $0<\gamma<2s$  and $V$ satisfies (\ref{mod-potential}), then the infimum in the variational problem
can be attained. That is, there exists $U\in \Sigma^s$ such that 
\be\label{VP2}
E(U)=d_M=\min\limits_{\{v\in \Sigma^s\big\vert \ \|v\|_2^2=M\}} E(v).
\ee
Moreover,  any minimizing sequence to \eqref{VP} must be relatively compact in $\Sigma^s$.
\end{proposition}

 \begin{proof} Firstly, we prove the variational problem (\ref{VP}) is well-defined. That is, $E(v)$ has  a lower bound
 in $\{v\in \Sigma^s\big\vert \ \|v\|_2^2=M\}$. 
 From Lemma \ref{l:ga-hardy-hartree}, we deduce that 
 \be\label{VP3}
 \begin{aligned}
 E(v) 
   &\geq\frac 12 \int (|\xi|^2+m^2)^s|\widehat{v}|^2d\xi+\frac 12 \int V |v|^2dx-  C\|v\|_{\dot{H}^s}^{\frac{\gamma}{s}}\
    \|v\|_2^{\frac{4s-\gamma}{s}}\\
&\geq\frac 12 \int |\xi|^{2s} |\widehat{v}|^2d\xi-C'_M\left(\int |\xi|^{2s} |\widehat{v}|^2d\xi\right)^{\frac{\gamma}{2s}}\\
     &\ge -C_{M}.
 \end{aligned}
 \ee 
In the last step, we have used the  elementary inequality $\frac12 X-C'_MX^{\ga/2s}\ge -C_{M}$ for all $X>0$ and some constants $C'_M, C_M>0$. 
Thus,  $E(v)$ is bounded from below and the infimum $d_M$ exists. 
 
Secondly, take any minimizing sequence $\{v_n\}$ of Problem (\ref{VP}) satisfying
\begin{equation}\label{VP4}
E(v_n)\rightarrow d_M \quad \textup{and} \quad
\|v_n\|_2^2\rightarrow M,  \quad \;\textup{as}\; n\rightarrow\infty.
\end{equation}
We see that there exists an $L$ such that for all $n\ge L$,
\be\label{VP6} E(v_n)<d_M+1.
\ee 
This, together with \eqref{VP3} implies  for all $n\ge L$, 
 \be\label{VP7}
 \frac 12\norm{v_n}_{\Sigma^s}^2 \le C_M \norm{v_n}_{\dot{H}^s}^{\ga/s} +d_M+1.
 \ee
Hence,  $\{v_n\}$ must be bounded in $\Sigma^s$ by virtue of the condition $\frac{\ga}{s}<2$. 
 
 Thirdly, from the boundedness of  $\{v_n\}$, we know there exists a subsequence 
 (still denoted by $\{v_n\}$) and  $U\in \Sigma^s$ 
such that
 \be\label{VP8} 
 v_n(x) \rightharpoonup U(x) \ \  {\rm weakly \ \ in }\ \ \ \ \Sigma^s.
 \ee
 By  the lower semi-continuity of norm $\Sigma^s$, we have
\begin{equation}\label{e:U-vn_sigma}
\|U \|_{\Sigma^s}^2\leq \liminf\limits_{n\rightarrow \infty} \|v_n \|_{\Sigma^s}^2\,.
\end{equation}  
Combining  \eqref{e:U-vn_sigma},   Lemma \ref{l:vn-hartree-ga}  and \eqref{VP4} we obtain
 that 
 $\|v_n\|_2^2\rightarrow \|U\|_2^2=M$ and  
 \[
 E(U)\leq \liminf_{n\to\iy}E(v_n)= d_M.  
 \]
 But from the definition of $d_M$, we must have $E(U)= d_M$. 
 That is, $U$ is a minimizer of  (\ref{VP}). 
To prove the statement on relative compactness, observe that the last argument shows 
\begin{align*}
E(U)= \lim_{n\to \iy}E(v_n).
\end{align*}
This, along with Lemma \ref{l:vn-hartree-ga} implies there exists a subsequence $\{v_{n_k}\}$ such that 
$\lim_{k\to\iy} \norm{v_{n_k}}_{\Sigma^s}= \norm{U}_{\Sigma^s}$.
Therefore, in view of \eqref{VP8},
 the strong convergence  $v_{n_k}\to U$ in $\Sigma^s$ follows. 
 \end{proof}
\begin{remark} The existence of ground states for FNLS with unbounded $V$ was obtained in \cite{Ch12a} 
via Nehari's manifold approach. The existence and symmetry of ground state solutions were studied in \cite{Wu14a} for Hartree equation with zero potential.
\end{remark}
  
Define the set
    \be\label{5.18} 
    S_M:=\{v\in \Sigma^s\big\vert \ v\  \text{ is the  minimizer of Problem  (\ref{VP}) }\}.
\ee 
From the Euler-Lagrange Theorem (see \cite{Cazenave2003}),   for any $v\in
S_M\subset \Sigma^s$, there exists $\omega \in \mathbb{R}$ such that
\be\label{5.0000} (-\triangle+m^2)^s v +V(x)v+\omega v- (\frac{1}{|x|^\gamma}\ast|v|^{2}) v=0.
\ee 
Moreover,   $u(t,x)=e^{i\omega t} v(x)$ is a standing
wave solution for  (\ref{1.1}). Thus, each $v$ in $S_M$ is called an orbit. 
  It is easy to check that for any $t>0$, if $v\in S_M$, then $e^{i\omega t} v(x)\in S_M$.  
Applying Proposition \ref{p:fnlsV}, we now
prove the following orbital stability  for (\ref{1.1}). More precisely, assuming Hypothesis \ref{h:lwp-blupaltern-conserv},  we
show that if the initial data   is close to an orbit $v\in S_M$,
then the solution   of evolution system
(\ref{1.1})-(\ref{1.2}) remains close to  $ S_M$, the set of ground states, for all time.

 \begin{theorem}[orbital stability of standing waves]\label{th:orb-stab_V}  
Let $M>0$ and $V$ satisfy (\ref{mod-potential}). Let  $0<s<N/2$ and $0<\gamma<2s$.  Then,  for arbitrary
$\varepsilon>0$, there exists $\delta>0$ such that  if the initial data $u_0$ in $\Sigma^s$ satisfies 
\be\label{5.20}
\inf\limits_{v\in S_M}\|u_0-v\|_{\Sigma^s}<\delta,
\ee 
it holds for  the solution $u$ of the Cauchy
problem (\ref{1.1})-(\ref{1.2}),
\be\label{5.20}
\inf\limits_{v\in S_M}\|u(t,x)-v(x)\|_{\Sigma^s}<\varepsilon 
\ee 
for all $t$, where $S_M$ is defined in (\ref{5.18}).
\end{theorem}
\begin{proof}   
We prove the theorem by contradiction  following  the standard method for NLS (see \cite{Cazenave2003,zhang2000a,Zhang00}).  
Let $u_0\in \Sigma^s$ and $u$ be the  unique solution for \eqref{1.1}.  
From the proof of Proposition \ref{p:fnlsV} and the conservation laws in Hypothesis \ref{h:lwp-blupaltern-conserv}, we obtain that for all $t$,   
\[
\frac12\norm{u(t,\cdot)}_{\Sigma^s}^2 \leq E(u_0)+C(\|u_0\|_2) \norm{u(t,\cdot)}_{\Sigma^s}^{\ga/s}. 
\]
This suggests   
 that the  $\Sigma^s$-norm of $u$ is uniformly bounded for all $t$. 

 Assume that  the conclusion in the theorem is false, then there 
exist  $\varepsilon_0>0$ and a  sequence of initial data
$\{u_{0,n}\}_{n=1}^{\infty}$ such that 
\be\label{5.22}
\inf\limits_{v\in S_M}\|u_{0,n}-v\|_{\Sigma^s}<\frac1n,
\ee
and there exists a sequence of time $\{t_n\}_{n=1}^{\infty}$   such that for all $n$,
 \be\label{5.23}
 \inf\limits_{v\in S_M}\|u_n(t_n,x)-v\|_{\Sigma^s}\geq\varepsilon_0.
 \ee
But from (\ref{5.22}), the conservation
laws and Lemma \ref{l:vn-hartree-ga}, it follows that there exist $w\in \Sigma^s$ and 
a subsequence of $\{u_n\}$ (still denoted by $\{u_n\}$)
such that as $n\rightarrow\infty$
\[\|u_n(t_n,x)\|_2^2 =\|u_{0,n}\|_2^2 \rightarrow \norm{w}_2^2=M\]
and
\[ E(u_n(t_n,x))=E(u_{0,n})\rightarrow E(w)=d_M.\]
Hence,  $\{u_n(t_n,\cdot)\}_{n=1}^{\infty}$ is a minimizing
sequence of  (\ref{VP}). According to Proposition \ref{p:fnlsV}, 
$w\in S_M$ is  a minimizer such that, when passing to a subsequence  if necessary,
 \be\label{5.24} 
 \|u_n(t_n,x)-w(x)\|_{\Sigma^s}\rightarrow 0\  \ {\rm as}\ \ n\rightarrow\infty.
 \ee
This   contradicts  (\ref{5.23}), which proves the theorem. 
\end{proof}
\begin{remark}\label{re:genV} The results in Theorem \ref{th:orb-stab_V} and Proposition \ref{p:fnlsV} continue to hold if $V$
satisfies the slightly more general condition \eqref{Potential}. 
To see this, it suffices to observe the following properties. 
Let $E_V(u):=E(u)=\frac12\norm{u}^2_{\Sigma^s}-\frac14\iint \frac{|u(x)|^2 |u(y)|^2}{|x-y|^\ga} dxdy$ be given as in \eqref{Ev:DesV_hartree}
and  $\Sigma^s_V:=\Sigma^s$ as in \eqref{e:Sigma^s_V}.  Let $C\ge c_0+c_1$ be a fixed constant. Then we have: 
\begin{enumerate}
\item $M(e^{i\theta}u)=M(u)$
\item $E_V(e^{i\theta}u)=E_V(u)$   
\item $e^{i\theta}S_M=S_M$
\item  $E_{V+C}(u)=E_V(u)+\frac{C}{2}M(u)$
\item The set of minimizers $S_M$ is independent of $V+C$ 
\item Denote by $u_V$ the solution of \eqref{1.1}-\eqref{1.2}, then  we have
$u_{V+C}(t,x)=e^{-itC}u_V(t,x)$.
\end{enumerate}
\end{remark} 

\begin{remark}\label{re:standardHartreeV} The case $s=1$, $\ga=2$ (mass-critical) 
was studied in \cite{HuangZhangLi13}, 
where the threshold for  the stability of standing waves for \eqref{1.1} are obtained. 
Earlier result on 
the l.w.p. and mass and energy conservation laws for \eqref{1.1} can be found in \cite{Cazenave2003}  for $s=1$, 
and $\ga\le 2$. 
Let  $Q_0$ be the unique radial positive ground state solution of \eqref{VP}, where $V=0$. 
The paper \cite[Theorem 4.1]{HuangZhangLi13} followed  a quite standard variational approach and proved that 
if  $s=1$, $\ga=2$ and  $M<\norm{Q_0}_2^2$,  
then there exist ground state solutions  
of the minimization problem \eqref{VP}, where $V=|x|^2$.  Moreover, these ground state solutions are orbitally stable.  

In the absence of $V$, when $\ga=2s$ ($0<s<1$ and $N\ge 2$),
Zhang and Zhu \cite{zhangzhu2015}  
 proved the orbital stability via profile decomposition method. 
They also showed strong instability of \eqref{1.1} by constructing blowup solutions with initial data arbitrarily close to $Q_0$ in $H^s$.
 \end{remark} 
 
\section{Appendix: Uniqueness of weak convergence} 
For the proof of Lemma \ref{l:vn-hartree-ga}   we need  the following.
\begin{proposition}\label{p:vn-wc-f-g} Let $s>0$ and $\Sigma^s\subset L^2$ be defined as in \eqref{e:Sigma^s_V} with $V$ satisfying \eqref{mod-potential}.
Let $\{v_n\}$ be a sequence in $\Sigma^s$. 
If $\{v_n\}$  converges weakly to $f$ in $\Sigma^s$ 
and $\{v_n\}$ converges weakly  to $g$ in $L^2$. Then $f$ is  identical to $g$ in $\Sigma^s\cap L^2$. 
\end{proposition}  

The domain of $H=H_{s,V}$ is given by 
\[
D(H)=\{  \phi\in L^2:  (-\De+m^2 )^s \phi\in L^2, V \phi\in L^2\}.\] 
It is easy to verify that $D(H)$ is a {complete} Hilbert subspace of $L^2$ with respect to the norm
\begin{align*}
\norm{\phi}_{D(H)}:=\left( \norm{(-\De+m^2 )^s \phi}^2_2+\norm{V\phi}_2^2  \right)^{1/2} .
\end{align*}
The form domain of $H$ is defined as: $v\in Q(H)\iff v\in L^2(\R^N)$ and 
\begin{align}
(-\De+m^2 )^{s/2} v\in L^2\quad and\quad V^{1/2}v\in L^2,
\end{align}
which is complete with respect to the norm $\sqrt{q(v,v)}=\norm{v}_{\Sigma^s}$ given in \eqref{v-Hs-norm}.  We have $D(H)\subset Q(H)$.
The form definition for $H$ is equivalent to the following weak form definition. 

\begin{definition} A  function $f\in \Sigma^s=H^s\cap D(\sqrt{V})$ is in $Q(H)$ 
if and only if for all $\phi\in H^{s}(\R^N)$, there exists an $h\in L^2\cap H^{-s}$ s.t. 
\begin{align*}
&\int_{\R^N} f(x)  H \phi  \, dx= \int h(x) \phi(x) dx \quad \forall \phi\in H^{s} \\
=&\int f(x) (-\De+m^2 )^s \phi\,dx+\int f (x) V \phi \,dx.
\end{align*}
\end{definition}
We  see that $Q(H)=D(H)=H^{2s}\cap D(V)$. 
 Since on $D(H)$ it holds 
 \[\la (-\De+m^2)^s\phi, \phi\ra+ \la V\phi ,\phi\ra\ge (m^{2s}+c_1)\la \phi,\phi\ra,  \]
we obtain that the spectrum $\sigma(H)\subset [c_1, \iy)$, which implies  $0$ is in the resolvent set for $H$.
Hence $H\inv=(H-0)\inv$: $L^2\to D(H)\subset \Sigma^s$ exists and is a continuous mapping. 

\begin{proof}[Proof of Proposition \ref{p:vn-wc-f-g}] For all $\phi\in D(H)\subset\Sigma^s\subset L^2$, $s> 0$, the inner product on $\Sigma^s$
 \begin{align*}
&\la f, \phi\ra_{\Sigma^s}=\lim_n \la v_n,\phi\ra_{\Sigma^s} \\
=&\lim_n\la (-\De+m^2 )^{s/2} v_n, (-\De+m^2 )^{s/2}\phi \ra_{L^2}+ \la V^{1/2}v_n,V^{1/2}\phi\ra_{L^2}\\
=&\lim_n q( v_n,\phi)=\lim_n\la v_n, H\phi\ra_{L^2}=\la g, H\phi\ra_2,
\end{align*}
where $H=H_{s,V}$ is a positive self-adjoint operator in $L^2$, cf. \eqref{v-Hs-norm} and the quadratic form defined there. 
On the other hand, observe that 
\begin{align*}
&\la f, \phi\ra_{\Sigma^s}=\la (-\De+m^2 )^{s/2}f, (-\De+m^2 )^{s/2}\phi \ra_{L^2}+ \la V^{1/2}v_n,V^{1/2}\phi\ra_{L^2}\\
=&\la  f, H\phi\ra_2.
\end{align*}
Now as functions (or distributions) it follows that $f=g$ in $L^2$ by taking $\phi= H\inv \vphi$ for all $\vphi\in L^2 $.
This proves Proposition \ref{p:vn-wc-f-g}.
\end{proof}

\begin{remark}  The proof here relies on 
 the fact that the spectrum of $H_{s,V}$ is bounded below by zero, whence one sees that  $H_V=(-\Delta+m^2)^s+V$ has an inverse that continuously maps $L^2$ onto $D(H)$. 
\end{remark}

\noindent{\bf Acknowledgment}  This paper was partially done when  Shihui Zhu visited  School of Mathematics at Georgia Institute of Technology, who  would like to thank the hospitality of the School of Mathematics. 
This work  was  supported by the National Natural Science Foundation of China grant No.~11501395, No.~11371267 and  Excellent Youth Foundation of Sichuan  Scientific Committee grant No.~2014JQ0039  in China. 

\vspace{0.4cm}
\bibliographystyle{amsplain}

\begin{thebibliography}{99}
\bibitem{BaoDong2011} W. Bao and X. Dong, Numerical methods for computing ground states and dynamics
of nonlinear relativistic Hartree equation for boson stars. {\em J. Comput. Phys.} {\bf 230} (2011), 5449--5469.


 \bibitem{Cazenave2003}T. Cazenave, {\em Semilinear Schr\"{o}dinger Equations.} Courant Lecture Notes in Mathematics {\bf 10}, NYU, CIMS, AMS 2003.

\bibitem{CazLion82} T. Cazenave and P. L. Lions, Orbital stability of standing waves for some nonlinear Schr\"odinger equations. {\em Commun. Math. Phys.} {\bf 85} (1982), 549--561.

\bibitem{Ch12a} M. Cheng, Bound state for the fractional Schr\"odinger equation with unbounded potential.  {\em J. Math. Phys.} {\bf 53} 
 (2012), 043507. 

 \bibitem{CHHO2014}  Y. Cho,   H. Hajaiej, G. Hwang and T. Ozawa,  On the orbital stability of fractional\  Schr\"{o}dinger equations. 
 {\em Comm. Pure Appl. Anal.} {\bf 13} (2014), 1267--1282.   





\bibitem{Feng16} B. Feng, {Sharp threshold of global existence and instability of standing wave for the Schr\"odinger-Hartree equation with a harmonic potential}. {\em Nonlinear Analysis}: Real World Applications
{\bf  31}  (2016),  132--145. 


\bibitem{FengZha18} B. Feng and  H. Zhang, {Stability of standing waves for the fractional Schr\"odinger-Hartree equation}.
{\em J.  Math. Anal.  Appl.} {\bf 460} (2018), 352--364.

\bibitem{FrohlichLenzmann2007} J. Fr\"{o}hlich  and E. Lenzmann,  Blowup for nonlinear wave equations describing boson stars. 
{\em Comm. Pure Appl. Math.}  {\bf 60} (2007), 1691--1705.

 \bibitem{Fukuizumi2001} R. Fukuizumi,  Stability  and instability of standing waves for the
nonlinear Schr\"{o}dinger equation with harmonic potential. {\em Discrete Contin. Dynam. Systems.} {\bf 7} (2001), 525--544.

 \bibitem{FukuizumiOhta2003} R. Fukuizumi and M. Ohta, Stability of standing waves for nonlinear
Schr\"{o}dinger equations with potentials. {\em Differential and Integral Eqs.} {\bf 16} (2003), 111--128.

  \bibitem{Guohuang2012} B. Guo and  D. Huang, 
  {Existence and stability of standing waves for nonlinear fractional Schr\"{o}dinger equations}. {\em J. Math. Phys.}  {\bf 53} (2012), 083702.

 \bibitem{guozhu18} Q. Guo and S. Zhu, Sharp threshold of blow-up and scattering for the fractional Hartree equation. 
 {\em J.  Differ. Eqns}. {\bf 264} (2018), 2802--2832.  

 \bibitem{HainzlLenzmannLewinSchlein2010}  C. Hainzl, E. Lenzmann, M. Lewin and B. Schlein, On blowup
 for time-dependent generalized Hartree-Fock Equations. {\em Ann. Henri Poincar\'{e}}, {\bf 11} (2010), 1023--1052.

\bibitem{HuangZhangLi13} J. Huang , J. Zhang  and X. Li,  
{Stability of standing waves for the $L^2$-critical Hartree equations with harmonic potential}. {\em Applicable Analysis}  {\bf 92} (2013),  2076--2083.  

\bibitem{KirZh16} K. Kirkpatrick and Y. Zhang, 
{Fractional Schr\"odinger dynamics and decoherence.} {\em Physica D}: Nonlinear Phenomena
 {\bf 332}  (2016),  41--54. 



\bibitem{LieYau87}  E. Lieb  and H.-T. Yau, 
{The Chandrasekhar theory of stellar collapse as the limit of quantum mechanics.}
 {\em Comm. Math. Phys}. {\bf 112} (1987), no. 1, 147--174.

 
 




 \bibitem{Oh89}  Y.-G. Oh, 
 {Cauchy problem and Ehrenfest's law of nonlinear Schr\"odinger equations with potentials}.
 {\em J.  Differ. Eqns}, {\bf  81} (1989),   255--274.	   


 \bibitem{RoseWeinstein1988} H.  Rose and M.  Weinstein, On the bound states of the nonlinear Schr\"{o}dinger equation with a linear potential. {\em Physica D.}  {\bf 30} (1988), 207--218.

\bibitem{SeSq14} S. Secchi  and M. Squassina, 
{Soliton dynamics for fractional Schr\"odinger equations.}
 {\em Applicable Analysis}  {\bf 93} (2014),  Issue 8.  1702--1729.  
  
 \bibitem{Tao2006} T. Tao, {\em Nonlinear Dispersive Equations: Local and Global Analysis}. CBMS regional conference series in mathematics, 2006.



\bibitem{Wu14a} D. Wu, 
{Existence and stability of standing waves for nonlinear fractional Schr\"odinger equations with Hartree type nonlinearity}. {\em J. Math. Anal.  Appl.}
 {\bf 411}  (2014),  530--542.  


 \bibitem{zhang2000a} J. Zhang, Stability of standing waves for the nonlinear Schr\"{o}dinger equations with unbounded potentials. {\em Z. Angew. Math.
Phys.}  {\bf 51} (2000), 489--503.

 \bibitem{Zhang00} J. Zhang, Stability of attractive Bose-Einstein condensates. {\em  J.  Stat. Phys.}  {\bf 101} (2000), 731--745.

 \bibitem{zhangzhu2015}  J. Zhang and S. Zhu, Stability and instability of standing waves  for the  nonlinear fractional Schr\"{o}dinger equation. 
 {\em  J. Dyn. Differ. Eqns.} {\bf 29} (2017), 1017--1030.  

 \bibitem{zhu2016} S. Zhu, On the blow-up solutions for the nonlinear fractional Schr\"{o}dinger equation. {\em J. Differ. Eqns}. {\bf 261} (2016), 1506--1531.
 

 \bibitem{zhu2017}  S. Zhu,  Existence of stable standing waves for the fractional Schr\"{o}dinger equations with combined nonlinearities. 
 {\em J. Evolu. Eqns.} {\bf 17}  (2017), 1003--1021.   
  \end{thebibliography}

 \end{document}